\documentclass[a4paper,11pt,twoside]{amsart}
\usepackage[english]{babel}
\usepackage[utf8]{inputenc}

\usepackage[a4paper,inner=3cm,outer=3cm,top=4cm,bottom=4cm,pdftex]{geometry}
\usepackage{fancyhdr}
\usepackage{titlesec}
\usepackage{enumerate}
\usepackage{color}
\usepackage{bold-extra}
\titleformat{\section}{\normalfont\scshape\centering}{\thesection}{1em}{}
\titleformat{\subsection}{\bfseries}{\thesubsection}{1em}{}
  
\usepackage{comment}
\usepackage{graphics}
\usepackage{aliascnt}
\usepackage[pdftex,citecolor=green,linkcolor=red]{hyperref}

\usepackage{amsmath}
\usepackage{amsfonts}
\usepackage{amssymb}
\usepackage{amsthm}
\usepackage{mathtools}

\newtheorem{theorem}{Theorem}[section]
\newtheorem{corollary}[theorem]{Corollary}

\newtheorem{lemma}[theorem]{Lemma}
\newtheorem{proposition}[theorem]{Proposition}
\theoremstyle{definition}

\newtheorem{remark}[theorem]{Remark}

\numberwithin{equation}{section}

\newcommand{\ord}{\textup{ord}}
\newcommand{\Gal}{\textup{Gal}}

\newcommand{\Li}{\textup{Li}}

\author{Olli J\"arviniemi}
\address{Department of Mathematics and Statistics, P.O. Box 68, 00014 Helsinki, Finland}
\email{olli.jarviniemi@helsinki.fi}

\title{Orders of algebraic numbers in finite fields}

\date{}

\begin{document}
\begin{abstract}
For an algebraic number $\alpha$ we consider the orders of the reductions of $\alpha$ in finite fields. In the case where $\alpha$ is an integer, it is known by the work on Artin's primitive root conjecture that the order is ``almost always almost maximal'' assuming the generalized Riemann hypothesis, but unconditional results remain modest. We consider higher degree variants under GRH. First, we modify an argument of Roskam to settle the case where $\alpha$ and the reduction have degree two. Second, we give a positive lower density result when $\alpha$ is of degree three and the reduction is of degree two. Third, we give higher rank results in situations where the reductions are of degree two, three, four or six. As an application we give an almost equidistribution result for linear recurrences modulo primes. Finally, we present a general result conditional to GRH and a hypothesis on smooth values of polynomials at prime arguments.
\end{abstract}

\maketitle

\section{Introduction}

Artin's famous primitive root conjecture asserts that for a given integer $a \neq -1$ not equal to a square, $a$ is a primitive root modulo infinitely many primes. While this is still unproven, we expect the density of such primes to be positive, as has been shown by Hooley \cite{hooley} under the generalized Riemann hypothesis (GRH). Furthermore, the density of primes $p$ with $\ord_p(a) \ge (p-1)/C$ should tend to $1$ as $C \to \infty$, as long as $|a| > 1$. In other words, the order of $a$ modulo $p$ should be ``almost always almost maximal''. This is again only known under GRH.

Essentially the best known unconditional result about the size of $\ord_p(a)$ is a square root bound proven by the following easy argument.

\begin{lemma}
\label{lem:global}
Let $a$ be an integer with $|a| > 1$. For all but $o(\pi(x))$ primes $p \le x$ one has $\ord_p(a) > \sqrt{p}/\log p$.
\end{lemma}

\begin{proof}
If for $p \le x$ a prime we have $\ord_p(a) \le \sqrt{p}/\log p$, then $p \mid T := \prod_{n \le \sqrt{x}/\log x} (a^n - 1)$, and so the product of such primes $p$ also divides $T$. However, $0 < |T| \ll a^{x/(\log x)^2}$, and hence there are at most $O(x/(\log x)^2)$ such primes $p$.
\end{proof}

One can slightly improve the lower bound by showing that for any function $\epsilon(p) \to 0$, almost all primes $p$ are such that $p-1$ has no divisor in $[p^{1/2 - \epsilon(p)}, p^{1/2 + \epsilon(p)}$] (see \cite{erdos-ram}). However, this is the best result we know: the result of Lemma \ref{lem:global} is not known with a lower bound $p^{1/2 + c}$ for any fixed $c > 0$.

A natural variant of the problem is to consider the orders of the reductions of algebraic integers in finite fields (of not necessarily prime size): Given an algebraic integer (or in general an algebraic number) $\alpha$, how large are the orders of the reductions of $\alpha$ in finite fields of characteristic $p$, as $p$ and the reduction map vary?

Note that there are usually several different ways to reduct $\alpha$ to an extension of $\mathbb{F}_p$. In elementary terms, the possibilities are given by the roots of the minimal polynomial of $\alpha$ in extensions of $\mathbb{F}_p$. In terms of prime ideals of rings of integers of number fields, the possibilities are given by considering the prime ideals $\mathfrak{p}$ of $O_{\mathbb{Q}(\alpha)}$ lying over $(p)$ and considering the image of $\alpha$ under the map $O_{\mathbb{Q}(\alpha)} \to O_{\mathbb{Q}(\alpha)}/\mathfrak{p}$, whose domain is a finite field of size $N(\mathfrak{p})$.

One could expect that, as in the integer case, the order of the reduction is again ``almost always almost maximal'' in an appropriate sense, no matter how the reductions are chosen as $p$ ranges over the primes.

As Lemma \ref{lem:global} is essentially the best known unconditional result in the integer case, it seems that essentially the best known result in the algebraic number situation is given by the proof of Lemma \ref{lem:global} adapted to the setting of algebraic numbers. 

However, in contrast to the integer case, in the general case the state of matters is far from satisfactory even under the assumption of GRH. We do know (under GRH) that a reduction of $\alpha$ is almost always of almost maximal order if the reduction is of degree $1$, corresponding to a result on the orders of $\mathbb{F}_p$ roots of a fixed polynomial $P \in \mathbb{Z}[x]$. (This is rather similar to the integer case.) However, in addition to this the only other situation we understand is the case where both $\alpha$ and its reduction have degree $2$ (corresponding to a result on the orders of the roots of a quadratic polynomial $P$ in $\mathbb{F}_{p^2}$ for those primes $p$ for which $P$ is irreducible modulo $p$).

While there have been several works on such topics in the case where both $\alpha$ and the reduction have degree $2$ (see \cite[Section 8.7.9]{moree} for references, in particular \cite{roskamUnit} and \cite{roskamQuad}), we have not found an explicit statement of the almost maximality of the order. Hence we state and prove this result below. The crucial ideas may be found in the work of Roskam \cite{roskamQuad}.

For technical reasons, it turns out that the problem becomes increasingly difficult as one increases the degree of the reduction of $\alpha$ and/or the degree of $\alpha$ itself, and it seems that such generalizations have received less attention. An application of such a result may be found in \cite{roskamLinRec}, and considerations focusing on the algebraic aspects of calculating relevant conjectural densities (rather than bounding the error terms) may be found in \cite{kitaoka1}, \cite{kitaoka2}. 

One of our main results is a positive (lower) density result in the case where $\deg(\alpha) = 3$ and the reduction is of degree $2$. We also consider the case where $\deg(\alpha)$ is arbitrary (again when the reduction is of degree $2$), but here we have to weaken our conclusion to a result in a higher rank situation. The benefit of having a higher rank is that the ``global argument'' of Lemma \ref{lem:global} may then be adapted to give stronger bounds (see \cite{matthews}, also \cite{agrawal-pollack}).

We furthermore consider more general situations where the degree of the reduction is at least three. Below we briefly describe what makes this problem harder than the original primitive root conjecture. See Section \ref{sec:overview} for more detailed discussion of the difficulties and our methods.

For the sake of discussion, consider the case where $\deg(\alpha) = 3$ and the reduction $\varphi(\alpha)$ has degree $3$ over $\mathbb{F}_p$, and try to determine the number of primes $p \le x$ for which $\varphi(\alpha)$ has maximal order $p^3 - 1$. Proceeding as Hooley \cite{hooley}, we could consider, for each prime $q \mid p^3 - 1 = (p-1)(p^2 + p + 1)$, whether $\varphi(\alpha)$ is a $q$th power in $\mathbb{F}_{p^3}$ or not. However, in the integer case considered by Hooley one has $q \mid p-1$, and so $q \le x$, while in our problem $q$ can be as large as $x^2$. As a result, when Hooley uses the argument of Lemma \ref{lem:global}, he may reduce to the region $q \le \sqrt{x}\log x$, which is then handled by (the GRH-conditional version of) the Chebotarev density theorem and the Brun-Titchmarsh inequality. By proceeding similarly, we are only able to reduce to $q$ at most $x^{3/2}\log x$, which is far from enough to be able to apply the Chebotarev density theorem or the Brun-Titchmarsh inequality.

Since the global argument of Lemma \ref{lem:global} is essentially the only method we know that is applicable for handling the contribution of large $q$ effectively, one may expect the problem to be rather difficult.

As the knowledge about the power residue conditions is very limited (boiling down essentially to the Chebotarev density theorem and the global argument), our line of attack is by being more careful with the congruence conditions. In practice this means restricting the (prime) divisors of $p^3 - 1 = (p-1)(p^2 + p + 1)$.

Unfortunately, the knowledge on the large prime factors of polynomial values (especially at prime arguments) is rather limited. Already in the simplest case of shifted primes we only know that the largest prime factor of $p-1$ is infinitely often as large as $p^{0.677}$ and as small as $p^{0.2961}$ \cite{baker-harman}, while we would expect corresponding bounds $\ge p^{1-\epsilon}$ and $\le p^{\epsilon}$ for any $\epsilon > 0$. Hence, such an approach does not lead unconditionally very far, but for completeness we do formulate a general result conditional to what we expect to be true below (Theorem \ref{thm:general}).

On the bright side, in the cases where the reduction $\varphi(\alpha)$ has degree equal to $k \in \{3, 4, 6\}$ the relevant polynomial $p^k - 1$ factorizes into linear polynomials and quadratics, and we are able to control the prime factors of $p^k - 1$ which lie in the range $[p^{1 - \epsilon}, p^{1 + \epsilon}]$. This is Proposition \ref{prop:mediumPrimeFactor} below; the author is very grateful to Jori Merikoski for providing a proof sketch. This result is enough to improve the bound given by the global argument of Lemma \ref{lem:global} by a small power of $p$ in a certain higher rank case. 

While this improvement is quite small, it leads to an interesting application. Niederreiter \cite[Theorem 4.1]{niederreiter} has shown that if the period of a linear recurrence in a finite field is large enough, then the linear recurrence is approximately equidistributed. More precisely, the result applies if the period of such a sequence of modulo $p$ is substantially larger than $p^{d/2 + 1}$, with $d$ being the depth of the recurrence. Note that the period is controlled by the least common multiple of the orders of the roots of the characteristic polynomial, connecting the result to our problem. Our lower bound on $\ord_p(\varphi(\alpha))$ is good enough to apply the result in the case $\deg(\alpha) = \deg(\varphi(\alpha)) \in \{4, 6\}$.

We remark that questions on how often a linear recurrence has a zero modulo $p$ (a conclusion much weaker than approximate equidistribution) have previously received some attention, especially in the case of linear recurrences of order two (see \cite[Section 8.4]{moree} for references). The author has in a recent preprint established a GRH-conditional result applicable for generic recurrences of any order \cite{jarviniemi}.  Previously such a result had been established under a hypothesis on almost maximality of the order almost always in the case $\deg(\alpha) = \deg(\varphi(\alpha))$ \cite{roskamLinRec}.

Below we present some notation and then state our results. An overview of the methods is given, and after giving the necessary preliminaries the detailed proofs follow. We freely assume GRH in all of our discussions.

\section{Notation and conventions}
\label{sec:notation}

The letter $p$ will always denote a prime and $\alpha$ denotes a non-zero algebraic number, whose degree over $\mathbb{Q}$ is denoted by $\deg(\alpha)$. The norm of $\alpha$ over $\mathbb{Q}$ is denoted by $N(\alpha)$.

For a given prime $p$, we let $\varphi = \varphi_p$ denote a homomorphism from a ring of algebraic numbers to $\mathbb{F}_{p^{\infty}}$. We usually omit the subscript (no confusion should arise). We will always assume that $\varphi$ is defined on all relevant algebraic numbers. For example, if a statement concerns the orders of $\alpha_1, \ldots , \alpha_n$, and in the proof we consider a reduction map $\varphi$, we assume $\varphi$ is defined on all of $\alpha_i$ and their conjugates. The specifics of which ring one defines $\varphi$ on are not important. In this example case one suitable choice is the set of numbers of the form $\beta/m$, where $\beta$ is an algebraic integer in the Galois closure of $\mathbb{Q}(\alpha_1, \ldots , \alpha_n)$ and $m$ is an integer not divisible by $p$.

We assume that $p$ is large enough so that such reductions indeed exist. (Note that if, for example, $\alpha = 1/2$, no such reduction exists for $p = 2$.) We also assume that $p$ is unramified, i.e. there are $\deg(\alpha_i)$ distinct reductions of $\alpha_i$ to extensions of $\mathbb{F}_p$ for all relevant (finitely many) $\alpha_i$. Furthermore, we assume that $p$ is large enough so that $\varphi_p(\alpha_i)$ is non-zero all relevant $\alpha_i$.

We denote the degree of $x \in \mathbb{F}_{p^{\infty}}$ over $\mathbb{F}_p$ by $\deg(x)$. The multiplicative order of an element $0 \neq x \in \mathbb{F}_{p^{\infty}}$ is denoted by $\ord_p(x)$. Note that $\ord_p(x)$ is trivially bounded by $p^{\deg(x)} - 1$.

\section{Results}
\label{sec:results}

For completeness, we first state a result well-known to experts, which handles the case $\deg(\varphi(\alpha)) = 1$ under GRH. (See e.g. \cite[Theorem 4]{erdos-ram} for a proof of $d(h) \to 1$ in the integer case -- the general case can be proven similarly. The existence follows by the method of Hooley \cite{hooley}.)

\begin{theorem}[The case $\deg(\varphi(\alpha)) = 1$]
\label{thm:degN1}
Assume GRH. Let $\alpha$ be a non-zero algebraic which is not a root of unity. For $h > 0$, let $d(h)$ denote the density of primes $p$ such that any degree one reduction $\varphi_p(\alpha)$ of $\alpha$ to $\mathbb{F}_{p^{\infty}}$ satisfies $\ord_p(\varphi_p(\alpha)) \ge (p-1)/h$. Then $d(h)$ exists and $d(h) \to 1$ as $h \to \infty$.
\end{theorem}

We then state the result for $\deg(\alpha) = \deg(\varphi(\alpha)) = 2$. The result is essentially due to Roskam \cite{roskamQuad}.

\begin{theorem}[The case $\deg(\alpha) = \deg(\varphi(\alpha)) = 2$]
\label{thm:deg22}
Assume GRH. Let $\alpha$ be an algebraic of degree $2$. Assume that the quotient of $\alpha$ and its conjugate is not a root of unity. Let $f(p) = p^2 - 1$ if $|N(\alpha)| \neq 1$ and $f(p) = 2(p+1)$ if $|N(\alpha)| = 1$. For $h > 0$, let $d(h)$ denote the density of primes $p$ such that for any reduction $\varphi_p(\alpha)$ of degree $2$ to $\mathbb{F}_{p^{\infty}}$ one has $\ord_p(\varphi_p(\alpha)) \ge f(p)/h$. Then $d(h)$ exists and $d(h) \to 1$ as $h \to \infty$.
\end{theorem}

(It is not hard to see that $\ord_p(\varphi(\alpha))$ divides $f(p)$. Also, while there are two different reduction maps $\varphi_{p, 1}, \varphi_{p, 2}$ for any (unramified) $p$, they are essentially equivalent, as $\varphi_{p, 1}(\alpha)$ and $\varphi_{p, 2}(\alpha)$ are conjugates over $\mathbb{F}_p$, and in particular their orders are equal. Similar remarks apply for the results below.)

For the case $\deg(\alpha) = 3, \deg(\varphi(\alpha)) = 2$ we have the following positive lower density result.

\begin{theorem}[The case $\deg(\alpha) = 3$, $\deg(\varphi(\alpha)) = 2$]
\label{thm:deg32}
Assume GRH. Let $\alpha$ be an algebraic number of degree $3$. Let $S$ denote the set of primes $p$ such that there is a reduction map $\varphi_p$ with $\deg(\varphi_p(\alpha)) = 2$. Assume $S$ is infinite (and thus of positive density). For $h > 0$, let $d(h)$ denote the lower density of $p \in S$ (with respect to $S$) such that $\ord_p(\varphi_p(\alpha)) \ge (p^2 - 1)/h$ for all reductions $\varphi_p(\alpha)$ of degree $2$. Then
$$\lim_{h \to \infty} d(h) \ge 1 - \log(\sqrt{6}) \ge 0.1041.$$
Furthermore, if $|N(\alpha)| = 1$, one has the stronger bound
$$\lim_{h \to \infty} d(h) \ge 1 - \log(\sqrt{3}) \ge 0.4506.$$
In particular, for some $h$ we have $d(h) > 0$.
\end{theorem}
Conjecturally the limit in question should be equal to one, but we are unable to prove this.

In the case of arbitrary $\deg(\alpha)$ (still with $\deg(\varphi(\alpha)) = 2$) we give two results in a higher rank situation. We first show that the order of the reductions of one of $\alpha_i$ is large. We then obtain a result with a smaller rank, at the cost of proving only that the order of some ``small'' product of $\alpha_i$ is large.

\begin{theorem}[The case $\deg(\varphi(\alpha)) = 2$, version one]
\label{thm:N2v1}
Assume GRH. Let $k \ge 3$ be a positive integer. Then, for $n = (2k-2)(2k-1) + 1$, the following holds:

Let $\alpha_1, \ldots , \alpha_n$ be algebraic numbers with $3 \le \deg(\alpha_i) \le k$ such that the numbers $\alpha_1, \ldots, \alpha_n$, together with their conjugates, are multiplicatively independent. Let $S$ denote the set of primes $p$ such that there is a reduction map $\varphi_p$ satisfying $\deg(\varphi_p(\alpha_i)) = 2$ for all $i$. Assume $S$ is infinite (and thus of positive density). For $h > 0$, let $d(h)$ denote the lower density of $p \in S$ such that for such reduction map $\varphi_p$ we have, for some $i$, $\ord_p(\varphi_p(\alpha_i)) \ge (p^2 - 1)/h$. Then $d(h) \to 1$ as $h \to \infty$.
\end{theorem}

\begin{theorem}[The case $\deg(\varphi(\alpha)) = 2$, version two]
\label{thm:N2v2}
Assume GRH. Let $k \ge 3$ be a positive integer. Then, for $n = 2k-1$, the following holds:

Let $\alpha_1, \ldots , \alpha_n$ be algebraic numbers with $3 \le \deg(\alpha_i) \le k$ such that the numbers $\alpha_1, \ldots, \alpha_n$, together with their conjugates, are multiplicatively independent. Let $S$ denote the set of primes $p$ such that there is a reduction map $\varphi_p$ satisfying $\deg(\varphi_p(\alpha_i)) = 2$ for all $i$. Assume $S$ is infinite (and thus of positive density). For $h > 0$, let $d(h)$ denote the lower density of $p \in S$ (with respect to $S$) such that for such reduction map $\varphi_p$ we have, for some integers $0 \le e_1, \ldots , e_n < 2k-1$, $\ord_p(\varphi_p(\alpha_1^{e_1} \cdots \alpha_n^{e_n})) \ge (p^2 - 1)/h$. Then $d(h) \to 1$ as $h \to \infty$.
\end{theorem}

We could relax the condition about the multiplicative independence a bit, but of course some conditions are needed (cf. Theorems \ref{thm:degN1} and \ref{thm:deg22}), not only to guarantee that the statement holds (so we have to assume e.g. that not all $\alpha_i$ are roots of unity) but also that we benefit from inspecting a higher rank situation (so we have to assume some sort of multiplicative independence condition between $\alpha_i$, where $i$ varies). 

However, we remark that the exact necessary assumptions are not as pretty as in the simpler cases of Theorems \ref{thm:degN1} and \ref{thm:deg22}. Here is one example: If $\alpha = \sqrt{1 + \sqrt{2}}$, then the conjugates of $\alpha$ are $\pm \sqrt{1 + \sqrt{2}}$ and $\pm i\sqrt{\sqrt{2} - 1}$ with the minimal polynomial being $x^4 - 2x^2 - 1$. One now sees that if one takes a large prime $p$ such that the reductions $\varphi(\alpha)$ are of degree $2$, then either the conjugate of $\varphi(\alpha)$ is $\pm \varphi(i \sqrt{1 + \sqrt{2}})$, in which case by the Frobenius endomorphism
$$\varphi(\alpha)^{4(p+1)} = \left(\sqrt{1 + \sqrt{2}}\right)^{8(p + 1)} = \left(\sqrt{1+\sqrt{2}} \cdot \pm i\sqrt{\sqrt{2} - 1}\right)^{4} = 1,$$
or the conjugate is $\varphi(-\alpha)$, in which case
$$\varphi(\alpha)^{2(p-1)} = \left(\varphi(-\alpha)/\varphi(\alpha)\right)^2 = 1.$$
Hence the order of $\varphi(\alpha)$ is bounded by $4p+4$. Of course, the issue is that $\alpha$ is multiplicatively dependent with its (global) conjugates, resulting in $\varphi(\alpha)$ being similarly multiplicatively dependent with its local conjugate $\varphi(\alpha)^p$, which in turn leads to small order.

The interested reader may determine the exact necessary hypotheses on $\alpha_i$ by inspecting the proofs, but for simplicity we have decided to not optimize our assumptions in this sense. Similar remarks apply for the theorems below.

We then move on to higher values of $\deg(\varphi(\alpha))$. We obtain a slight improvement over the easy bound (given by the global argument of Lemma \ref{lem:global}) in a higher rank case when $\deg(\varphi(\alpha)) \in \{3, 4, 6\}$.

\begin{theorem}[The case $\deg(\varphi(\alpha)) \in \{3, 4, 6\}$]
\label{thm:Nk}
Assume GRH. Write $g(3) = 2, g(4) = 3$ and $g(6) = 4$. Let $k \in \{3, 4, 6\}$ and $D \in \mathbb{Z}_+$ be given. There exists a constant $c_D$ depending on $D$ (and $k$) such that the following holds for any $\epsilon > 0$ and $n = \lfloor 1/\epsilon \rfloor$.

Let $\alpha_1, \ldots , \alpha_{n}$ be algebraic numbers which, together with their conjugates, are multiplicatively independent, and for which $\deg(\alpha_i) \le D$. Let $S$ denote the set of primes $p$ such that there is a reduction map $\varphi_p$ satisfying $\deg(\varphi_p(\alpha_i)) = k$ for all $i$. Assume $S$ is infinite. Let $d$ denote the lower density (with respect to $S$) of primes $p \in S$ such that for any such $\varphi_p$ there exists some $1 \le i \le n$ such that
$$\ord_p(\varphi_p(\alpha_i)) \ge p^{g(k) + \epsilon}.$$
Then $d \ge 1 - c_D\epsilon$.

In particular, $d > 0$ for small enough $\epsilon$ and $d \to 1$ as $\epsilon \to 0$.
\end{theorem}
One could in principle give $c_D$ explicitly in terms of $D$, but we have not pursued this direction, since the resulting bounds are not particularly strong and obtaining such bounds would be rather laborous.

By combining the above result with a theorem of Niederreiter \cite[Theorem 4.1]{niederreiter} on the behavior of linear recurrences modulo primes we obtain the following.

\begin{theorem}[Almost equidistribution of linear recurrences modulo primes]
\label{thm:linRec}
Assume GRH. Let $k \in \{4, 6\}$. For each large enough $n$ there exist constants $c_1(n), c_2(n) > 0$ with $c_1(n) \to 1$ as $n \to \infty$ such that the following holds.

Let $\alpha_1, \ldots , \alpha_n$ be algebraic numbers of degree $k$ which, together with their conjugates, are multiplicatively independent. Let $S$ denote the set of primes $p$ such that there is a reduction map $\varphi_p$ satisfying $\deg(\varphi_p(\alpha_i)) = k$ for all $i$. Assume $S$ is infinite. For each $i$, let $a_{i, 1}, a_{i, 2}, \ldots$ be a linearly recursive sequence of rationals, not all zero, such that the characteristic polynomial of $(a_{i, m})_{m \in \mathbb{Z}}$ is the minimal polynomial of $\alpha_i$. Let $d$ denote the lower density of primes $p \in S$ (with respect to $S$) such that for some $1 \le N \le n$ we have
$$|\{1 \le j \le T_{i, p} : a_{i, j} \equiv r \pmod{p}\}| = \frac{T_{i, p}}{p}\left(1 + O(p^{-c_2(n)})\right)$$
for all $r \pmod{p}$, where $T_{i, p}$ denotes the period of $(a_{i, m})$ modulo $p$. Then $d \ge c_1(n)$.
\end{theorem}

Note that Theorems \ref{thm:N2v1}, \ref{thm:N2v2}, \ref{thm:Nk} and \ref{thm:linRec} imply results that hold for ``almost all'' algebraic numbers in a certain sense. For example, we show the following corollary of Theorem \ref{thm:linRec}.

\begin{corollary}[Almost equidistribution modulo primes of almost all linear recurrences]
\label{cor:linRec}
Assume GRH. Let $k \in \{4, 6\}$. For any small enough $\epsilon > 0$ there exist a constant $c_{\epsilon} > 0$ and a number field $K_{\epsilon}$ such that the following holds.

Let $\alpha$ be an algebraic number of degree $k$ such that $\alpha$ is multiplicatively independent with its conjuages and the Galois closure of $\mathbb{Q}(\alpha)$ is linearly disjoint with $K_{\epsilon}$. Assume that the set $S$ of primes $p$ such that the reductions of $\alpha$ to $\mathbb{F}_{p^{\infty}}$ have degree $k$ is infinite. Let $a_1, a_2, \ldots$ be a linearly recursive sequence of rationals, not all zero, such that the characteristic polynomial of the sequence is the minimal polynomial of $\alpha$. Then there exists a subset $S'$ of $S$ of relative upper density $1 - \epsilon$ such that for any $p \in S'$ and any $r \pmod{p}$ one has
$$|\{1 \le j \le T_p : a_j \equiv r \pmod{p}\}| = \frac{T_p}{p}\left(1 + O(p^{-c_{\epsilon}})\right),$$
where $T_p$ is the period of $(a_n)$ modulo $p$.
\end{corollary}

We are not able to determine a specific $K_{\epsilon}$ which has this property. Compare this with the fact that we know that there are at most two primes for which Artin's primitive root conjecture fails, but we are not able to pinpoint for which primes it potentially may fail \cite{heath-brown}. However, one could in principle calculate an upper bound for the degree of $K_{\epsilon}$ in terms of $\epsilon$ by making the dependence between $D$ and $c_D$ in Theorem \ref{thm:Nk} explicit (similarly as we have an upper bound for the number of primes which may fail Artin's conjecture).

We further note that while we can prove that $S'$ has large upper density in Corollary \ref{cor:linRec}, our methods are insufficient for showing that $S'$ has large (or even positive) lower density. See Remark \ref{rem:lowerDensity} for an explanation.

Finally, we note that a general result with arbitrary $\deg(\alpha), \deg(\varphi(\alpha))$ follows under an assumption on smooth values of (cyclotomic) polynomials at prime arguments (and GRH).

\begin{theorem}[The general case]
\label{thm:general}
Assume GRH. Let $\alpha$ be an algebraic number of degree $k \ge 3$. Assume $\alpha$ is multiplicatively independent with its conjugates. Let $d \ge 2$ be given.

Let $S$ denote the set of primes $p$ such that
\begin{itemize}
\item The largest prime divisor of $p^d - 1$ is less than $p^{1/2k}$.
\item There is at least one reduction $\varphi(\alpha)$ of $\alpha$ of degree $d$.
\end{itemize}
Assume $S$ has positive lower density. For $h > 0$, let $d(h)$ denote the lower density of $p \in S$ (with respect to $S$) such that for all reductions $\varphi_p(\alpha)$ of $\alpha$ of degree $d$ one has $\ord_p(\varphi(\alpha)) \ge (p^d - 1)/h$. Then $d(h) \to 1$ as $h \to \infty$.
\end{theorem}

We have not optimized the assumptions (as current technology is far from such results). One could, for example, relax the condition to that for any $d' \mid d$, any prime divisor of $\Phi_{d'}(p)$ is either at most $p^{1/2k}$ or at least $p^{\phi(d') - 1/2 + \epsilon}$. Note that this in particular implies that instead of proving that $\Phi_{d'}(p)$ has no large prime factors it would also be sufficient to show that $\Phi_{d'}(p)$ has a very large prime factor (but again, we are far from showing such results).

We remark that unconditional results relying only on Lemma \ref{lem:global} admit generalizations to our setting, with the strongest results being obtained for $\deg(\varphi(\alpha)) \le 2$. (And as we have already explained, essentially the best unconditional results one can expect are those following from Lemma \ref{lem:global}.) However, such modifications do not require any substantial new ideas, and hence we do not focus on such aspects. Instead we refer to \cite{agrawal-pollack}, where the applicability of the global argument is demonstrated.

\section{Overview of the method}
\label{sec:overview}

\subsection{Hooley's argument for Artin's primitive root conjecture}

We find it easiest to explain our arguments by comparing them to the integer case and focusing on the differences. Hence, while Hooley's proof of Artin's conjecture under GRH is well-known today, we find it worthwhile to briefly go through the argument here.

Let $a$ be an integer, $|a| > 1$, $a$ not a square. Consider the primes $p \le x$ such that $a$ is a primitive root modulo $p$.

For a prime $q$, let $C_q = C_q(a)$ denote the condition that $p \equiv 1 \pmod{q}$ and $a$ is a $q$th power modulo $p$. An integer $a \not\equiv 0 \pmod{p}$ is a primitive root modulo $p$ if and only if no condition $C_q$ is satisfied. For a given $q$ the density of primes $p$ satisfying $C_q$ is given by the Chebotarev density theorem. (In the generic case this density is $1/q(q-1)$, but this is not true, for example, if $a$ is a $q$th power of an integer.)

One may hence try to calculate the density of primes $p$ such that $a$ is a primitive root modulo $p$ by an inclusion-exclusion argument. This, however, does not work due to the accumulation of error terms.

To go around this issue, proceed as follows: Let $f(x) = \frac{1}{6}\log x$. We first show, under GRH, that the number of primes $p \le x$ not satisfying any of the conditions $C_q, q \le f(x)$ is (approximately) what it should be. We then show, again under GRH, that the number of primes $p \le x$ satisfying at least one condition $C_q, q > f(x)$ is $o(\pi(x))$.

The contribution of primes $q \le f(x)$ is handled by he Chebotarev density theorem and inclusion-exclusion. It is non-trivial (and interesting) to actually calculate the arising density, but we do not focus here on such aspects.

The contribution of primes $q > f(x)$ is handled as follows. For each $f(x) \le q \le \sqrt{x}/(\log x)^{100}$ one bounds the number of $p \le x$ satisfying $C_q$ by the Chebotarev density theorem (with GRH used to bound the error term). For $\sqrt{x}/(\log x)^{100} \le q \le \sqrt{x}(\log x)^{100}$ one simply estimates the number of primes $p \le x, p \equiv 1 \pmod{q}$ by Brun-Titchmarsh. Finally, for $\sqrt{x}(\log x)^{100} \le q \le x$ one applies the global argument mentioned in the introduction (Lemma \ref{lem:global}).

Hence, Hooley's argument can be seen to consist of four steps.
\begin{itemize}
\item[(i)] Consider the primes $q \le f(x)$. Method: Chebotarev density theorem and inclusion-exclusion.
\item[(ii)] Consider the primes $q \in [f(x), \sqrt{x}/(\log x)^{100}]$. Method: Chebotarev density theorem.
\item[(iii)] Consider the primes $q \in [\sqrt{x}/(\log x)^{100}, \sqrt{x}(\log x)^{100}]$. Method: Brun-Titchmarsh.
\item[(iv)] Consider the primes $q \in [\sqrt{x}(\log x)^{100}, x]$. Method: The global argument.
\end{itemize}
(Note that trivially all $p \le x$ fail the conditions $C_q$ for all $q > x$.)

Note that by letting $f(x)$ grow slowly enough one could work out step (i) unconditionally without the need of GRH. Hence this step is rather uninteresting from an analytic point of view, so our focus is on the steps (ii)--(iv).

In contrast to the Chebotarev density theorem and the Brun-Titchmarsh equality utilized steps (i)--(iii), where each individual $q$ is considered separately, in step (iv) one consider all conditions $C_q, q \ge \sqrt{x}(\log x)^{100}$ simultaneously. This distinction is why we call the method of Lemma \ref{lem:global} the global argument.

\subsection{The case $\deg(\varphi(\alpha)) = 1$ -- Theorem \ref{thm:degN1}}

Consider then the setup of Theorem \ref{thm:degN1}. The result is best viewed in terms of the order of the reduction of $\alpha$ under the maps $O_{\mathbb{Q}(\alpha)} \to O_{\mathbb{Q}(\alpha)}/\mathfrak{p}$ for prime ideals $\mathfrak{p}$ of $O_{\mathbb{Q}(\alpha)}$. Note that almost all prime ideals of $O_{\mathbb{Q}(\alpha)}$ are of degree $1$, that is, their norms are primes.

The analytic methods required in the proof are largely the same as those discussed above. Step (iii) needs no change. Step (iv) requires some minor cosmetic changes, but the proof is otherwise the same. In step (ii) one applies the Chebotarev density theorem with the base field being $\mathbb{Q}(\alpha)$ instead of $\mathbb{Q}$.

One easily sees why the assumption that $\alpha$ is not a root of unity is needed.

\subsection{The case $\deg(\alpha) = \deg(\varphi(\alpha)) = 2$ - Theorem \ref{thm:deg22}}

The situation is drastically different when one considers reductions not of degree $1$ over $\mathbb{F}_p$, as very few prime ideals of $\mathbb{Q}(\alpha)$ are of degree over $1$. Hence we take another point of view.

Consider the case $\deg(\alpha) = 2$. Let $S$ be the set of primes $p$ such that the reductions of $\alpha$ to $\mathbb{F}_{p^{\infty}}$ have degree $2$ (and the reductions are distinct). 

Let $\varphi(\alpha)$ be either of the reductions of $\alpha$ to $\mathbb{F}_{p^2}$ with $p \in S$. Let $C_q$ be the condition that $\varphi(\alpha)$ is not a $q$th power for $q \mid p^2 - 1$. (Note that in this case the conjugate of $\varphi(\alpha)$ is not a $q$th power either.)

Luckily, the expression $p^2 - 1$ factorizes as $(p-1)(p+1)$, so similarly as in Hooley's argument, it suffices to consider the conditions $C_q$ with $q$ roughly at most $x$. Some care is still needed. We consider the conditions $C_q, q \mid p-1$, and $C_q, q \mid p+1$ separately.

The case $q \mid p-1$ essentially reduces to the integer case: If $\varphi(\alpha)$ is a $q$th power, then 
$$N(\alpha)^{(p-1)/q} = N_{\mathbb{F}_{p^2}/\mathbb{F}_p}(\varphi(\alpha))^{(p-1)/q} = (\varphi(\alpha)^{p+1})^{(p-1)/q} = 1,$$ 
and hence the norm $N(\alpha)$ is a $q$th power. (Here we slightly abuse notation by viewing $N(\alpha) \in \mathbb{Z}$ as an element of $\mathbb{F}_{p^2}$.)

In the case $q \mid p+1$ one notes that
$$(\varphi(\alpha)^p/\varphi(\alpha))^{(p+1)/q} = 1,$$
so if $\alpha'$ is the conjugate of $\alpha$ (different from $\alpha$), the reduction of $\alpha'/\alpha$ is a perfect $q$th power in $\mathbb{F}_{p^2}$. One then applies the ideas of steps (ii), (iii) and (iv) to bound the contribution of such $q$. The situation is not very different from the one in Theorem \ref{thm:degN1}.

However, there is one crucial technical detail to be noted in step (ii). In Hooley's argument, one applies the Chebotarev density theorem to bound the number of primes $p \le x$ which split in $\mathbb{Q}(\zeta_q, a^{1/q})/\mathbb{Q}$, and similarly in the setup of Theorem \ref{thm:degN1} one considers the split primes in $\mathbb{Q}(\zeta_q, a^{1/q})/\mathbb{Q}(\alpha)$. In our case case the relevant Artin symbol condition is not a splitting condition: the relevant condition is that $p$ does not split in $\mathbb{Q}(\alpha)/\mathbb{Q}$, and that the orbits of $\zeta_q, \alpha^{1/q}$ and $\alpha'^{1/q}$ under the Artin symbol of $p$ with respect to
$$\mathbb{Q}(\zeta_q, \alpha^{1/q}, \alpha'^{1/q})/\mathbb{Q}$$
have size $2$. This is an issue, as the GRH-conditional Chebotarev density theorem gives better error terms for smaller conjugacy classes (and hence in particular one has the best error terms for splitting conditions).

The way around this, introduced by Roskam in \cite{roskamQuad}, is to show that if $p$ has Artin symbol given by the condition above, then $p$ splits in some relatively large subfield of $\mathbb{Q}(\zeta_q, \alpha^{1/q}, \alpha'^{1/q})$. This results in a small enough error term.

Again, from this argument one sees where the necessary conditions of Theorem \ref{thm:deg22} arise from.

\subsection{The case $\deg(\alpha) = 3, \deg(\varphi(\alpha)) = 2$ -- Theorem \ref{thm:deg32}}
\label{sec:overview:32}

Take then an algebraic number $\alpha$ of degree $3$, and consider those primes $p$ such that some reductions of $\alpha$ to $\mathbb{F}_{p^{\infty}}$ have degree $2$ (and the reductions are distinct). Let the set of such primes $p$ be $S$. Let $\varphi(\alpha)$ denote a degree $2$ reduction of $\alpha$ to, and let $C_q$ be the condition that $\varphi(\alpha)$ is not a $q$th power for $q \mid p^2 - 1$. 

The conditions $q \mid p-1$ again are treated similarly as in the integer case.

If $q \mid p+1$, then, as described in the previous subsection, we encounter a problem in step (ii) when applying the GRH-conditional Chebotarev density theorem: our Artin symbol condition is not a splitting condition, and hence  the error term is too large. Furthermore, it seems that in general there is no large subfield $K'_q$ of
$$K_q := \mathbb{Q}(\zeta_q, \alpha_1^{1/q}, \alpha_2^{1/q}, \alpha_3^{1/q}),$$
where $\alpha_i$ are the conjugates of $\alpha$, such that if $p \in S$ fails the condition $C_q, q \mid p + 1$, then $p$ splits in $K'_q$, so one cannot apply Roskam's idea here.

Hence we have to take a different route, utilizing other tools to obtain a smaller error term. Our argument may be expressed as follows:
\begin{itemize}
\item[(i)] Consider the primes $q \le f(x)$.
\item[(ii, a)] Consider the primes $q \in [f(x), x^{1/6}/(\log x)^{100}]$. Method: GRH-conditional Chebotarev density theorem.
\item[(ii, b)] Consider the primes $q \equiv 2 \pmod{3}$, $q \in [x^{1/6}/(\log x)^{100}, x^{1/4}/(\log x)^{100}]$. Method: GRH-conditional Chebotarev density theorem with known cases of Artin's conjecture on L-functions.
\item[(iii, a)] Consider the primes $q \equiv 1 \pmod{3}, q \in [x^{1/6}/(\log x)^{100}, x^{1/2}/(\log x)^{100}]$. Method: The Bombieri-Vinogradov theorem.
\item[(iii, b)] Consider the primes $q \equiv 2 \pmod{3}$, $q \in [x^{1/4}/(\log x)^{100}, x^{1/2}/(\log x)^{100}]$. Method: The Bombieri-Vinogradov theorem.
\item[(iii, c)] Consider the primes $q \in [x^{1/2}/(\log x)^{100}, x^{1/2}(\log x)^{100}]$. Method: The Brun-Titchmarsh inequality.
\item[(iv)] Consider the primes $q \in [\sqrt{x}(\log x)^{100}, x]$. Method: The global argument.
\end{itemize}
Compared to the previous proofs the new ingredients are those in steps (ii, b) and (iii, a), (iii, b). The GRH-conditional Chebotarev density theorem is strong enough to only cover the region $[f(x), x^{1/6}/(\log x)^{100}]$, but if one further assumes Artin's conjecture for L-functions, one has an even smaller error term which allows one to cover $q$ as large as $x^{1/4 - \epsilon}$. In our case the field extension $K_q$ is relatively nice (more precisely, its Galois group is ``almost nilpotent''), so that Artin's conjecture has been proven for $K_q$, at least when $q \equiv 2 \pmod{3}$. We take advantage of this in step (ii, b).

However, this still leaves a quite large gap between the regions covered by the Chebotarev density theorem and the global argument. While the Brun-Titchmarsh inequality can be used to treat short regions of form $[x^{1/2}/(\log x)^{100}, x^{1/2}(\log x)^{100}]$, for larger regions one encounters a loss of density: a positive proportion of primes $p$ are such that $p+1$ has a prime divisor in the interval $[x^{1/4}, x^{1/2}]$ (at least conjecturally). For this reason we use the stronger Bombieri-Vinogradov theorem in steps (iii, a) and (iii, b). This still leads to a loss of density, but fortunately this loss is strictly smaller than one.

We remark that the idea of imposing a loss of density has been used before at least in \cite{nongkynrih}, where it is shown that Artin's primitive root conjecture for $a = 2$ follows by a hypothesis on the zeros of L-functions weaker than GRH, at the cost of obtaining a smaller (lower) density of primes $p$ for which $a$ is a primitive root modulo $p$ than what is expected to hold.

\subsection{The case $\deg(\varphi(\alpha)) = 2$ -- Theorems \ref{thm:N2v1} and \ref{thm:N2v2}}

We are not able to adapt the argument of the previous subsection to the case where the reduction is of degree $2$ but $\alpha$ itself is of arbitrary degree, as the region the Chebotarev density theorem is able to handle becomes smaller as $\deg(\alpha)$ increases (and hence a similar procedure would result in a loss of density greater than $1$).

Hence we weaken our conclusions to a result on a higher rank situation. The advantage is that now the global argument is able to handle a larger region.

Theorem \ref{thm:N2v1} is obtained by a rather standard application of a higher rank variant of the global argument attributed to Matthews \cite{matthews}. Theorem \ref{thm:N2v2} is obtained by applying the global argument in a slightly different way. The latter result was inspired by a recent work of Agrawal and Pollack \cite{agrawal-pollack}, though in our case the argument is more straightforward, as we assume GRH and thus may use the Chebotarev density theorem to handle $q$ less than some small power of $x$.

\subsection{The case $\deg(\varphi(\alpha)) \in \{3, 4, 6\}$ -- Theorem \ref{thm:Nk}}

The nature of the problem changes substantially when the reduction is of degree greater than two. As we already explained in the introduction, the problem lies in that for primes $p \le x$ we have to consider the conditions $C_q$ for $q$ as large as $x^2$ (or even larger if $\varphi(\alpha)$ has very high degree).

Consider, for simplicity, the case $\deg(\varphi(\alpha)) = 3$ (though the method adapts to $\deg(\varphi(\alpha))$ equal to $4$ or $6$ as well). Again, the conditions $C_q, q \mid p-1$ are easily treated, so we consider $C_q, q \mid p^2 + p + 1$.

The global argument only suffices to handle $q > x^{3/2 + \epsilon}$, though weakening the conclusion to a higher rank analogue allows one to treat the region $[x^{1+\epsilon}, x^2]$ for any fixed $\epsilon > 0$.

We treat the region $[x^{1-\epsilon}, x^{1+\epsilon}]$ by showing that $p^2 + p + 1$ has relatively rarely a prime divisor $q$ in this interval. The method of proof is by exponential sums, more precisely to Weyl sums over roots of quadratic congruences.

In fact, we need a slightly stronger result, as proving that $p^2 + p + 1$ has no prime divisor in $[x^{1-\epsilon}, x^{1+\epsilon}]$ and that the order is at least $p^{2 - \epsilon}$ does not imply that $\ord_p(\varphi(\alpha))$ is greater than $p^{2+\epsilon}$ -- it could very well be the case that $p^2 + p + 1$ is divisible by two primes $q_1, q_2 \approx x^{1/2}$ and that $C_{q_1}, C_{q_2}$ both are satisfied. Hence, we actually show that $p^2 + p + 1$ is rarely divisible by an $x^c$-rough number $Q \in [x^{1-\epsilon}, x^{1+\epsilon}]$ for any fixed $c > 0$. We still need to show that $C_q, q < x^c$ almost never fail, but this follows from the GRH-conditional Chebotarev density theorem.

If one could show that $p^2 + p + 1$ rarely has any divisor in the range $[x^{1-\epsilon}, x^{1+\epsilon}]$, then one could obtain the result of Theorem \ref{thm:Nk} unconditionally. (Results of this type do hold for shifted primes \cite[Theorem 2]{erdos-ram}, \cite[Theorem 1.6]{ford}.)

\section{Preliminaries}
\label{sec:prel}

We mention a couple of facts we will be using throughout the proofs.

The GRH-conditional Chebotarev density theorem states that for any Galois extension $K/\mathbb{Q}$ and conjugacy class $C \subset \Gal(K/\mathbb{Q})$ one has
\begin{align}
\label{eq:cheb}
|\{p \le x : \left(\frac{K/\mathbb{Q}}{p}\right) = C\}| = \Li(x)\frac{|C|}{[K : \mathbb{Q}]} + O\left(\frac{|C|}{[K : \mathbb{Q}]}\sqrt{x}\log d_K  + |C|\sqrt{x}\log x\right),
\end{align}
where $d_K$ is the discriminant of $K/\mathbb{Q}$ (see the works of Lagarias-Odlyzko \cite{lo} and Serre \cite{serre}). 

In our applications the contribution of the discriminant is not too large: By \cite[Proposition 13]{perucca-sgobba-tronto} if $d_{K_m}$ is the discriminant of
\begin{align*}
K_m = K(\zeta_m, \alpha_1^{1/m}, \ldots , \alpha_n^{1/m}),
\end{align*}
where $\alpha_i \in K$ are non-zero and not roots of unity and $K$ is a given number field, then we have, for $m$ large,
\begin{align*}
\log d_{K_m} \ll [K_m : \mathbb{Q}]\log m.
\end{align*}
The implied constant depends on $K$ and $\alpha_i$ but is independent of $m$.

We note that a trivial bound for $[K_m : \mathbb{Q}]$ is $\ll \phi(m)m^n$ (where again the implied constant may depend on $K$ and $\alpha_i$). This bound is essentially tight \cite[Theorem 9]{perucca-sgobba-tronto}: if $\alpha_i$ are multiplicatively independent, then $[K_m : \mathbb{Q}] \gg \phi(m)m^n$.

\section{The case $\deg(\alpha) = \deg(\varphi(\alpha)) = 2$ -- Proof of Theorem \ref{thm:deg22}}

Our proof is a modification Roskam's \cite{roskamQuad} argument (see in particular Section 4 there). We first consider the case where $|N(\alpha)| \neq 1$ -- the case $|N(\alpha)| = 1$ then follows with minor modifications, as we explain at the end of the proof.

Let $S$ be the set of primes $p$ such that $\deg(\varphi(\alpha)) = 2$, i.e. $S$ consists of the primes $p$ inert in $\mathbb{Q}(\alpha)$. In the proof we only consider $p \in S$. Let $H$ be a parameter tending to infinity arbitrarily slowly. Let $h = H!$. For each prime $q$, let $\ell(q)$ be the smallest power of $q$ not dividing $h$, and let $C_q = C_q(h)$ be the condition
\begin{center}
$p \in S$ satisfies $p^2 \equiv 1 \pmod{\ell(q)}$ and $\varphi(\alpha)$ is a perfect $\ell(q)$th power in $\mathbb{F}_{p^2}$.
\end{center}
Note that $\ell(q) = q$ for all but very small $q$, and that if $p \in S$ satisfies none of the conditions $C_q$, then $(p^2 - 1)/h \mid \ord_p(\varphi(\alpha))$. We show that the density of primes $p \in S$ not satisfying any of the conditions $C_q$ is $1$ for any $H \to \infty$, which then implies the result.

If $p$ satisfies $C_q$, then $p$ is either such that $\ell(q) \mid p-1$ and $N(\alpha)$ is a perfect $\ell(q)$th power in $\mathbb{F}_p$ , or $p$ is such that $\ell(q) \mid p+1$ and $\varphi(\alpha_2/\alpha_1)$ is a perfect $\ell(q)$th power in $\mathbb{F}_{p^2}$, where $\alpha_1 = \alpha$ and $\alpha_2$ are the conjugates of $\alpha$. Let these conditions be $C_q^1$ and $C_q^2$, respectively.

Let $C_q(x)$ (resp. $C_q^1(x)$, $C_q^2(x)$) denote the number of primes $p \le x$ satisfying the condition $C_q$ (resp. $C_q^1$, $C_q^2$). 

We perform the steps (i)--(iii) in our proof sketch, showing that
$$\sum_{q \le \sqrt{x}(\log x)^{100}} C_q(x) = o(\pi(x)).$$
 
Step (i). Note that
$$\sum_{q \le H} C_q(x) \le \sum_{q \le H} |\{p \le x : p^2 \equiv 1 \pmod{\ell(q)}\}|,$$
and by Brun-Titchmarsh the contribution of this sum is $o(\pi(x))$.
 
Step (ii). We to show that
$$\sum_{H < q \le \sqrt{x}/(\log x)^{100}} C_q^i(x) = o(\pi(x))$$
for $i \in \{1, 2\}$.

We first consider the easier case $i = 1$. In this case $N(\alpha)$ is a $q$th power modulo $p$ with $p \equiv 1 \pmod{q}$, so $p$ splits in
$$K_q^1 := \mathbb{Q}(\zeta_q, N(\alpha)^{1/q}).$$
Hence, the GRH-conditional Chebotarev density theorem gives
$$C_q^1(x) = \Li(x)\frac{1}{[K_q^1 : \mathbb{Q}]} + O(\sqrt{x}(\log x)),$$
and since $N(\alpha) \neq \pm 1$, one has $[K_q^1 : \mathbb{Q}] \gg q(q-1)$. (Recall the preliminaries from Section \ref{sec:prel}.) This proves the required result for $i = 1$.

We then consider the more difficult case $i = 2$. Let
$$K_q^2 := \mathbb{Q}(\zeta_q, \alpha_1^{1/q}, \alpha_2^{1/q}).$$
One sees that $p$ satisfies $C_q^2$ if and only if the Artin symbol $\left(\frac{K_q^2/\mathbb{Q}}{p}\right)$ maps $\zeta_q \to \zeta_q^{-1}$, $\alpha_1 \to \alpha_2 \to \alpha_1$ and is such that the order of $\alpha_j^{1/q}$ under the Artin symbol is $2$ for $j \in \{1, 2\}$. Hence, if $p \in C_q^2$, then the order of the Artin symbol in $G := \Gal(K_q^2/\mathbb{Q})$ is two. 

Let $C$ be the set of maps which are one of the following forms:

Form 1: There is an integer $a$ such that
$$\zeta_q \to \zeta_q^{-1},$$
$$\alpha_1^{1/q} \to \zeta_q^a\alpha_2^{1/q},$$
and
$$\alpha_2^{1/q} \to \zeta_q^a\alpha_1^{1/q}.$$

Form 2: There is an integer $b$ such that
$$\zeta_q \to \zeta_q,$$
$$\alpha_1^{1/q} \to \zeta_q^b\alpha_1^{1/q}$$
and
$$\alpha_2^{1/q} \to \zeta_q^b\alpha_2^{1/q}.$$

\begin{lemma}
$C$ is a normal subgroup of $G$.
\end{lemma}

\begin{proof}
One may check this by a direct computation, which we omit here.
\end{proof}

Note now that if $p$ satisfies $C_q^2$, then $\left(\frac{K_q^2/\mathbb{Q}}{p}\right) \subset C$. Hence, if we let $F$ denote the fixed field of $C$, then $p$ splits in $F$. Furthermore, $|C| \le 2q$, so
$$[F : \mathbb{Q}] \ge \frac{[K_q^2 : \mathbb{Q}]}{2q} \gg q^2.$$
Hence the Chebotarev density theorem gives the required result.

Step (iii). Note that
\begin{align*}
\sum_{\sqrt{x}/(\log x)^{100} \le q \le \sqrt{x}(\log x)^{100}} C_q(x) \le \\
\sum_{\sqrt{x}/(\log x)^{100} \le q \le \sqrt{x}(\log x)^{100}} |\{p \le x : p \equiv 1 \pmod{q}\}| + |\{p \le x : p \equiv -1 \pmod{q}\}| = \\
o(\pi(x))
\end{align*}
by Brun-Titchmarsh.

Step (iv). Finally, we prove that the number of primes $p \le x$ satisfying at least one of $C_q^i, q > \sqrt{x}(\log x)^{100}$ is $o(\pi(x))$ for $i \in \{1, 2\}$. Let $T_i$ be the product of such primes $p$.

Note that
$$T_1 \mid \prod_{n \le \sqrt{x}/(\log x)^{100}} N(\alpha)^n - 1,$$
from which one obtains $\log T_1 \ll x/(\log x)^{200} = o(\pi(x)),$ implying the result for $i = 1$.

Similarly,
$$T_2 \mid N\left(\prod_{n \le \sqrt{x}/(\log x)^{100}} \left(\frac{\alpha_2}{\alpha_1}\right)^n - 1\right) = \prod_{n \le \sqrt{x}/(\log x)^{100}} N(\alpha_2^n/\alpha_1^n - 1),$$
from which one gets $\log T_2 = o(\pi(x))$, implying the result for $i = 2$.

\begin{remark}
In the case $N(\alpha) = \pm 1$ one omits the consideration of the conditions $C_q^1$ and proceeds otherwise similarly. In step (ii) one notes that in the definition of $C$ the numbers $a$ and $b$ must satisfy certain conditions due to $\alpha_2 = \pm \alpha_1^{-1}$. This leads to $|C| = O(1)$. Hence, the fixed field $F$ is of degree $[F : \mathbb{Q}] \gg [K_q : \mathbb{Q}] \gg q^2$, from which step (ii) is completed as before.
\end{remark}

\section{The case $\deg(\alpha) = 3$, $\deg(\varphi(\alpha)) = 2$ -- Proof of Theorem \ref{thm:deg32}}

Let $\alpha$ be an algebraic number of degree $3$ with conjugates $\alpha_1 = \alpha, \alpha_2, \alpha_3$, let $S$ be as in the theorem and let $K = \mathbb{Q}(\alpha_1, \alpha_2, \alpha_3)$. The condition that $S$ is infinite is equivalent with $[K : \mathbb{Q}] = 6$. For $p \in S$, let $\varphi(\alpha)$ denote a reduction of $\alpha$ to $\mathbb{F}_{p^2}$ having degree $2$ over $\mathbb{F}_p$.

Define $h, H, q$ and $\ell(q)$ similarly as in the proof of Theorem \ref{thm:deg22}. Let $C_q^1 = C_q^1(h)$ denote the condition
\begin{center}
$p \in S$ satisfies $p \equiv 1 \pmod{\ell(q)}$ and $\varphi(\alpha)$ is a perfect $\ell(q)$th power in $\mathbb{F}_{p^2}$
\end{center}
and let $C_q^2 = C_q^2(h)$ be the condition
\begin{center}
$p \in S$ satisfies $p \equiv -1 \pmod{\ell(q)}$ and $\varphi(\alpha)$ is a perfect $\ell(q)$th power in $\mathbb{F}_{p^2}$.
\end{center}
Again, let $C(x)$ denote the number of primes $p \le x$ satisfying the condition $C$.

We first note that the number of primes $p \le x$ satisfying at least one of the conditions $C_q^1$ is $o(\pi(x))$. Indeed: Let $p$ be a prime. Label $\alpha_1, \alpha_2, \alpha_3$ such that $\deg(\varphi(\alpha_3)) = 1$. Now
$$\varphi(\alpha_1\alpha_2) = N(\alpha)/\varphi(\alpha_3)$$
in $\mathbb{F}_{p^2}$. Hence $p$ satisfies $C_q^1$ if and only if
\begin{center} 
$p \in S$ satisfies $p \equiv 1 \pmod{\ell(q)}$ and the degree $1$ reduction of $N(\alpha)/\alpha$ to $\mathbb{F}_{p^2}$ is a perfect $\ell(q)$th power in $\mathbb{F}_p$.
\end{center}
Hence, the number of $p \in S, p \le x$ satisfying $C_q^1$ for some $q$ is the same as the number of $p \in S, p \le x$ such that $\ord_p(N(\alpha)/\alpha)$ is not divisible by $(p-1)/h$. We now conclude by applying Theorem \ref{thm:degN1} to the algebraic number $N(\alpha)/\alpha$. (Note that $N(\alpha)/\alpha$ is not a root of unity: If it were, then its norm $N(\alpha)^3/N(\alpha) = N(\alpha)^2$ would be $1$, so that $N(\alpha) = \pm 1$ and $\alpha$ would be a root of unity, contradicting the assumptions.)

Hence, from now on, we only consider the conditions $C_q^2$. We again first consider the case $|N(\alpha)| \neq 1$ and then explain the modifications needed to treat the case $|N(\alpha)| = 1$.

We proceed as explained in Section \ref{sec:overview:32}, showing that that $$\sum_{1 \le \sqrt{x}(\log x)^{100}} C_q^2(x) \le \pi(x)\log \sqrt{6} + o(\pi(x))$$
(steps (i)--(iii)) and that the number of $p \le x$ satisfying at least one of $C_q^2, q > \sqrt{x}(\log x)^{100}$ is $o(\pi(x))$ (step (iv)). Steps (i), (iii, c) and (iv) are executed similarly as in the proof of Theorem \ref{thm:deg22}, so we do not cover them in detail.

We first consider the steps (ii, a) and (ii, b). One notes that if $p$ satisfies $C_q$, then the Artin symbol of $p$ with respect to
$$K_q := \mathbb{Q}(\zeta_q, \alpha_1^{1/q}, \alpha_2^{1/q}, \alpha_3^{1/q})$$
is such that for any $\sigma \in \left(\frac{K_q/\mathbb{Q}}{p}\right)$ one has
\begin{itemize}
\item $\sigma(\zeta_q) =  \zeta_q^{-1}$
\item There exist distinct $i, j \in \{1, 2, 3\}$ such that $\sigma(\alpha_i) = \alpha_j$ and $\sigma(\alpha_j) = \alpha_i$.
\item For these $i, j$ one has $\sigma(\sigma(\alpha_i^{1/q})) = \alpha_i^{1/q}$ and $\sigma(\sigma(\alpha_j^{1/q})) = \alpha_j^{1/q}$.
\end{itemize}
(Conversely, any $\sigma$ satisfying the above properties belongs to the Artin symbol.) Let $C$ denote the set of such maps. Then $C$ is a union of conjugacy classes.

\begin{lemma}
\label{lem:sizeC}
We have $|C| \ll q^2$. If $|N(\alpha)| = 1$, then $|C| \ll q$.
\end{lemma}

\begin{proof}
One can choose $i$ and $j$ in $3$ ways. There are at most $q$ ways to choose the image of $\alpha_i^{1/q}$, each choice being of the form $\zeta_q^t\alpha_j^{1/q}$. By $\sigma(\sigma(\alpha_i^{1/q})) = \alpha_i^{1/q}$, the image of $\alpha_i^{1/q}$ uniquely determines the image of $\alpha_j^{1/q}$. Furthermore, there are at most $q$ options for the image of $\alpha_k^{1/q}$, $k \in \{1, 2, 3\} \setminus \{i, j\}$.

In the case $|N(\alpha)| = 1$ we note that given the image of $\alpha_i^{1/q}$ and $\alpha_j^{1/q}$ and that $\zeta_q \to \zeta_q^{-1}$, the image of $\alpha_k^{1/q} = N(\alpha)^{1/q}/(\alpha_i^{1/q}\alpha_j^{1/q})$ has $O(1)$ possible values.
\end{proof}

The GRH-conditional Chebotarev density theorem now gives
$$C_q^2(x) = \frac{\Li(x)|C|}{[K_q : \mathbb{Q}]} + O(\sqrt{x}|C|\log x) \ll \frac{x}{q^2} + O(\sqrt{x}q^2\log x).$$
It follows that the contribution of $H < x < x^{1/6}/(\log x)^{100}$ is $o(\pi(x))$.

For (ii, b) we improve upon the error term via the following lemma.

\begin{lemma}
\label{lem:q=5mod6}
Assume $q$ is a large enough prime (in terms of $\alpha$) satisfying $q \equiv 2 \pmod{3}$. Then the Artin L-functions associated to $K_q/\mathbb{Q}$ are entire.
\end{lemma}

\begin{proof}
This follows from \cite[Lemma 3.10]{wong}. Let $N = \mathbb{Q}(\zeta_q, \alpha_1, \alpha_2, \alpha_3)$. For $q$ large enough we have
$$\Gal(N/\mathbb{Q}) \cong \Gal(\mathbb{Q}(\zeta_q)/\mathbb{Q}) \times \Gal(\mathbb{Q}(\alpha_1, \alpha_2, \alpha_3)/\mathbb{Q}).$$
The irreducible representations of cyclotomic extensions have dimension $1$ and those of the $S_3$ Galois group have dimension at most $2$. Note then that $\Gal(K_q/\mathbb{Q})/\Gal(N/\mathbb{Q})$ is of size dividing $q^3$ and hence nilpotent. Finally, since
$$|\Gal(K_q/\mathbb{Q})| = [K_q : \mathbb{Q}] \mid 6(q-1)q^3,$$
we have $36 \nmid |\Gal(K_q/\mathbb{Q})|$ if $q \equiv 2 \pmod{3}$.
\end{proof}

By \cite[Corollary 3.7]{murty-saradha}, GRH together with the above lemma gives the improved error term
$$C_q^2(x) = \frac{\Li(x)|C|}{[K_q : \mathbb{Q}]} + O(\sqrt{x}\sqrt{|C|}\log x) = \frac{x}{q^2} + O(\sqrt{x}q \log x).$$
It follows that
$$\sum_{\substack{x^{1/6}/(\log x)^{100} \le q \le x^{1/4}/(\log x)^{100} \\ q \equiv 2 \pmod{3}}} C_q^2(x) = o(\pi(x)).$$

For steps (iii, a) and (iii, b), note that $p \in S$ if and only if the minimal polynomial of $\alpha$ has exactly one root modulo $p$, which is equivalent with the discriminant $D$ of $P$ not being a square modulo $p$. Hence, by the law of quadratic reciprocity $p \in S$ if and only if $p$ lies in the union of certain congruence classes modulo $4|D|$. Now, if $\pi_S(x; q, -1)$ counts the number of primes $p \in S \cap [1, x]$ with $p \equiv -1 \pmod{q}$, then the Bombieri-Vinogradov theorem one obtains
\begin{align*}
\sum_{\substack{x^{1/4}/(\log x)^{100} \le q \le \sqrt{x}/(\log x)^{100} \\ q \equiv 2 \pmod{3}}} C_q^2(x) \le \\
\sum_{\substack{x^{1/4}/(\log x)^{100} \le q \le \sqrt{x}/(\log x)^{100} \\ q \equiv 2 \pmod{3}}} \pi_S(x; q, -1) = \\
\pi(x)d(S)\sum_{\substack{x^{1/4}/(\log x)^{100} \le q \le \sqrt{x}/(\log x)^{100} \\ q \equiv 2 \pmod{3}}} \frac{1}{q} + O(x/(\log x)^2) = \\
\pi(x)d(S)\frac{\log 2}{2} + o(\pi(x)).
\end{align*}
Similarly,
\begin{align*}
\sum_{\substack{x^{1/6}/(\log x)^{100} \le q \le \sqrt{x}/(\log x)^{100} \\ q \equiv 1 \pmod{3}}} C_q^2(x) \le \\
\pi(x)d(S)\frac{\log 3}{2} + o(\pi(x)).
\end{align*}

It follows that the number of $p \le x, p \in S$ satisfying some of the conditions $C_q^2$ is at most $\pi(x)d(S)\log \sqrt{6} + o(\pi(x))$, as desired.

\begin{remark}
If $|N(\alpha)| = 1$, then by Lemma \ref{lem:sizeC} we have $|C| \ll q$, and hence the Chebotarev density theorem may be applied in (ii, a) up to $x^{1/4}/(\log x)^{100}$ and in (ii, b) up to $x^{1/3}/(\log x)^{100}$. In steps (iii, a) and (iii, b) the Bombieri-Vinogradov theorem then gives density losses $(\log 2)/2$ and $(\log 3/2)/2$, resulting in a total loss of density equal to $\log \sqrt{3}$.
\end{remark}

\begin{remark}
If $|N(\alpha)| = 1$, one obtains density loss $\log 2 < 1$ in a simpler way by omitting step (ii, b) and using Bombieri-Vinogradov for the whole region $[x^{1/4 + o(1)}, x^{1/2 + o(1)}]$. However, in the case $|N(\alpha)| \neq 1$ such an approach would lose $\log 3 > 1$ density, thus failing to give a non-trivial result.
\end{remark}

\begin{remark}
In steps (iii, a) and (iii, b) the more careless estimate $C_q^2(x) \le |\{p \le x : p \equiv -1 \pmod{q}\}|$ is not enough, as we lose the factor $d(S) = 1/2$. In order to apply Bombieri-Vinogradov for the primes $p \in S, p \equiv -1 \pmod{q}$ we used the fact that $S$ is characterized by congruence conditions. However, the analogue of the Bombieri-Vinogradov theorem restricted to primes satisfying a given Artin symbol condition proven in \cite{murty-petersen} would have been enough in this situation as well.
\end{remark}

\begin{remark}
It seems plausible that one may adapt the argument to the case $\deg(\alpha) = 4, \deg(\varphi(\alpha)) = 2$ under some additional assumptions on $\alpha$. Step (ii) requires attention. Initially the conjugacy class is of size $|C| \ll q^3$, so the GRH-conditional error term only handles $q \le x^{1/8 - o(1)}$. However, some cases of Artin's conjecture are known. If one establishes the Artin conjecture for all $K_q$ (and not just for, say, $q \equiv 2 \pmod{3}$), then one may treat $q \le x^{1/5 - o(1)}$, and the Bombieri-Vinogradov argument handles $x^{1/5 - o(1)} \le q \le x^{1/2 - o(1)}$ with a loss of $\log 5/2 < 1$ (assuming one may handle the issue mentioned in the previous remark). Furthermore, if one assumes $|N(\alpha)| = 1$, one has $|C| \ll q^2$, further shrinking the errors. However, we have not pursued the limits of this approach, as it seems that one cannot handle the general case $\deg(\alpha) = 4, \deg(\varphi(\alpha)) = 2$ with these ideas alone.
\end{remark}

\section{The case $\deg(\varphi(\alpha)) = 2$ -- Proof of Theorems \ref{thm:N2v1} and \ref{thm:N2v2}}

Let $\alpha_1, \ldots , \alpha_n$ and $S$ be as in the theorem. Let $K$ be the Galois closure of $\mathbb{Q}(\alpha_1, \ldots , \alpha_n)$. Define $h, H, q$ and $\ell(p)$ as before, and define $C_{q, i}^j, j \in \{1, 2\}, 1 \le i \le n$ as in the proof of Theorem \ref{thm:deg32} for $\alpha_i$.

Similarly as in the proof of Theorem \ref{thm:deg32}, the number of primes $p \le x$ satisfying at least one of the conditions $C_{q, i}^1$ is $o(\pi(x))$ as $H \to \infty$, so we may again focus on the conditions $C_{q, i}^2$.

The proof consists of the following steps.
\begin{itemize}
\item[(i)] Consider the primes $q \le H$. Method: Brun-Titchmarsh.
\item[(ii)] Consider the primes $q \in [H, x^{1/2k}/(\log x)^{100}]$. Method: GRH-conditional Chebotarev density theorem.
\item[(iii)] Consider the primes $q \in [x^{1/2k}/(\log x)^{100}, x^{1/2k}(\log x)^{100n}]$. Method: The Brun-Titchmarsh inequality.
\item[(iv)] Consider the primes $q \in [x^{1/2k}(\log x)^{100n}, x]$. Method: The global argument.
\end{itemize}
In the light of the proofs of the previous theorems, it is clear how to perform steps (i)--(iii).

The method for handling step (iv) is a bit different in the situations of Theorems \ref{thm:N2v1} and \ref{thm:N2v2}. We first state the following consequence of the global argument.

\begin{lemma}
\label{lem:globalHigh}
The number of $p \le x$ such that for some $q > x^{1/2k}(\log x)^{100}$, $C_{q, i}^2$ is satisfied for at least $2k-1$ values of $i$ is $o(\pi(x))$.
\end{lemma}

\begin{proof}
Fix some $1 \le i_1, \ldots , i_{2k-1} \le n$ and let $T$ denote the product of $p$ for which there exists $q > x^{1/2k}(\log x)^{100}$ such that $C_{q, i}^2$ is satisfied for $i = i_1, \ldots , i_{2k-1}$. Then $T$ divides
$$\prod_{\substack{\sigma \in \Gal(K/\mathbb{Q}) \\ \sigma \neq \text{id}}} \prod_{m_1, \ldots , m_{2k-1} \le x^{1/2k}/(\log x)^{100}} N(\alpha_{i_1}^{m_1} \cdots \alpha_{i_{2k-1}}^{m_{2k-1}} - \sigma(\alpha_{i_1})^{m_1} \cdots \sigma(\alpha_{i_{2k-1}})^{m_{2k-1}}).$$
The right hand side is non-zero by the assumption of multiplicative independence and its logarithm is $O(x/(\log x)^{100})$. Hence $\log T = o(\pi(x))$, and the claim follows by summing over all subsets $\{i_1, \ldots , i_{2k-1}\}$ of $\{1, 2, \ldots , n\}$.
\end{proof}

Theorem \ref{thm:N2v1} follows: By the lemma, almost all $p \in S$ have the property that for any $q \mid p+1, q \ge x^{1/2k}(\log x)^{100}$ the condition $C_{q, i}^2$ is satisfied for at most $2k-2$ values of $i$. Since there are at most $2k-1$ such large prime divisors $q$ of $p+1$ and $n > (2k-2)(2k-1)$, for at least one $i$ the condition $C_{q, i}^2$ is not satisfied for any such large $q$. This concludes the proof of Theorem \ref{thm:N2v1}.

Theorem \ref{thm:N2v2} follows as well: Note that since $n = 2k-1$, by the lemma the number of $p \le x, p \in S$ such that for some $q \mid p+1, q > x^{1/2k}(\log x)^{100}$ each of $\varphi(\alpha_1), \ldots , \varphi(\alpha_n)$ is a $q$th power in $\mathbb{F}_{p^2}$ is $o(\pi(x))$. Discard those $p$. Now, proceeding as in the proof of \cite[Theorem 1.2]{agrawal-pollack}, we see that for any prime $q > x^{1/2k}(\log x)^{100}$, among the $N^n - 1$ numbers
$$\varphi(\alpha_1^{e_1} \cdots \alpha_n^{e_n}), 0 \le e_i < N, e_1 + \ldots + e_n \neq 0,$$ 
at most $N^{n-1} - 1$ are $q$th powers in $\mathbb{F}_{p^2}$. Hence, as long as
$$N^n - 1 > (2k-1)(N^{n-1} - 1)$$
there is at least one number of the form
$$\varphi(\alpha_1^{e_1} \cdots \alpha_n^{e_n}), 0 \le e_i < N, e_1 + \ldots + e_n \neq 0$$
which is not a $q$th power for any $q \mid p+1, q > x^{1/2k}(\log x)^{100}$, as there are at most $2k-1$ such large prime divisors $q$ of $p+1$. The above inequality is satisfied for $N = 2k-1$.

\section{The general case -- Proof of Theorem \ref{thm:general}}

The proof of Theorem \ref{thm:general} is similar to the proofs of the other theorems. The region $q \le x^{1/2k}/(\log x)^{100}$ is handled by the GRH-conditional Chebotarev density therorem, the region $x^{1/2k}/(\log x)^{100} \le q \le x^{1/2k}(\log x)^{100}$ is handled by Brun-Titchmarsh, and if $p > x/(\log x)$, say, then by assumption on the smoothness of $p^d - 1$ the conditions imposed by primes $q > x^{1/2k}(\log x)^{100}$ are trivially satisfied. We omit the details.

\section{The case $\deg(\varphi(\alpha)) \in \{3, 4, 6\}$ -- Proof of Theorem \ref{thm:Nk}}

The additional ingredient to the preceeding discussion is the following result, a proof sketch kindly supplied to us by Jori Merikoski.

\begin{proposition}
\label{prop:mediumPrimeFactor}
Let $P \in \mathbb{Z}[x]$ be an irreducible quadratic with positive leading coefficient and negative discriminant. Let $c > 0$ be fixed. Let $d_c(\epsilon)$ denote the upper density of primes $p$ such that $P(p)$ is divisible by some number $m \in [p^{1-\epsilon}, p^{1+\epsilon}]$ such that all prime factors of $m$ are at least $p^c$. Then $d_c(\epsilon) \ll_{c} \epsilon$.
\end{proposition}
Already the statement that for some $\epsilon > 0$ a positive lower density of $n$ are such that $n^2 + 1$ has no prime divisor in $[n^{1-\epsilon}, n^{1+\epsilon}]$ is non-trivial. The proposition gives $d(\epsilon) \to 0$ with linear decay when $n$ ranges over the primes (which is heuristically optimal up to the implied constant). The proof of the above proposition also shows the corresponding result when $n$ is not restricted to the primes. We have for simplicty restricted to polynomials with negative discriminants, but the proof should go through in general (see the remarks in the proof).

We first show how to deduce Theorem \ref{thm:Nk} from here before proving the proposition. Fix $k \in \{3, 4, 6\}$ and $\epsilon > 0$, and let $\alpha_1, \ldots , \alpha_n$ and $S$ be as in the theorem and let $\varphi$ be a reduction map with $\deg(\varphi(\alpha_i)) = k$ for all $i$. Define $h, H, q$ and $\ell(p)$ as before and let $K$ be the Galois closure of $\mathbb{Q}(\alpha_1, \ldots , \alpha_n)$. For each divisor $d$ of $k$ and $1 \le i \le n$, let $C_{q, i}^d$ be the condition
\begin{center}
a prime $p \in S$ satisfies $\Phi_d(p) \equiv 0 \pmod{\ell(q)}$ and $\varphi(\alpha_i)$ is a perfect $\ell(q)$th power in $\mathbb{F}_{p^k}$,
\end{center}
where $\Phi_d$ is the $d$th cyclotomic polynomial. (The first six cyclotomic polynomials are $\Phi_1(x) = x-1, \Phi_2(x) = x+1, \Phi_3(x) = x^2 + x + 1, \Phi_4(x) = x^2 + 1, \Phi_5(x) = x^4 + x^3 + x^2 + x + 1$ and $\Phi_6(x) = x^2 - x + 1$.)

Let
$$c = \frac{1}{2\max(\deg(\alpha_1), \ldots , \deg(\alpha_n))},$$
and let $S'$ be the set of primes $p \in S, p \le x$ such that for any $d \mid k, d > 2$ the number $\Phi_d(p)$ is not divisible by any integer
$$m \in [h^{-1}x^{1-\epsilon}, x^{1+1/(n+1)}(\log x)^{100n}]$$
whose smallest prime factor is greater than $p^c$. By Proposition \ref{prop:mediumPrimeFactor} the lower density of $S'$ is at least
$$d(S) - C\epsilon$$
for some constant $C$ depending only on $c$. (Recall that $h$ tends to infinity arbitrarily slowly and $n \approx 1/\epsilon$.)

The Chebotarev density theorem may be applied to treat the region $q < x^c/(\log x)^{100}$. Indeed, one sees that $C_{q, i}^d$ is satisfied only if the Artin symbol of $p$ with respect to
$$K_{q, i}^d := \mathbb{Q}(\zeta_{\ell(q)}, \alpha_{i, 1}^{1/\ell(q)}, \ldots , \alpha_{i, \deg(\alpha_i)}^{1/\ell(q)}),$$
where $\alpha_{i, j}$ are the conjugates of $\alpha_i$, is such that the order of $\zeta_q$ is $d$ and for some $j$ the order of both $\alpha_{i, j}$ and $\alpha_{i, j}^{1/q}$ are $k$. It is not difficult to see that the set of automorphisms of $K_{q, i}^d$ with these properties form a conjugacy class $C$, and furthermore that $|C| \ll q^{-2}[K_{q, i}^d : \mathbb{Q}] \ll q^{\deg(\alpha_i) - 1}$.

Note that the contribution of the short interval $[x^c/(\log x)^{100}, x^c]$ may furthermore be treated by Brun-Titchmarsh.

We then note that the global argument is sufficient to show that for any $d \mid k$, there are only $o(\pi(x))$ primes $p \in S'$ with the property that there exist primes $q_1, \ldots , q_m$ such that $C_{q_j, i}^d$ is satisfied for all $1 \le j \le m, 1 \le i \le n$ and $q_1q_2 \cdots q_m > x^{1_{d > 2} + 1/(n+1)}(\log x)^{100n}$. This is essentially Lemma \ref{lem:globalHigh}

This already handles the conditions $C_{q, i}^d$ for $d \in \{1, 2\}$: we may assume that $\epsilon$ is small enough, so that $n$ is large enough in terms of $c$, so that $x^c$ is larger than $x^{1/(n+1)}(\log x)^{100n}$.

Hence, if for some prime $p \in S'$ one had $\ord_p(\varphi(\alpha_i)) \le x^{f(k) + \epsilon}$ for all $i$ then, neglecting the $o(\pi(x))$ exceptional primes described above, for some $d \mid k, d > 2$ there exist primes $q_1, \ldots , q_m$ dividing $\Phi_d(p)$ with $q_j > x^c$ and
$$q_1q_2 \cdots q_m \in [h^{-1}x^{1 - \epsilon}, x^{1 + 1/(n+1)}(\log x)^{100n}],$$
contradicting the definition of $S'$. This concludes the proof.

\begin{proof}[Proof of Proposition \ref{prop:mediumPrimeFactor}]
The main idea is transforming the problem to that of Weyl sums for quadratic congruences, a method developed by Hooley (see e.g. \cite{hooleyAverage}, Theorem 1 in particular), though we will use a result of Duke, Friedlander and Iwaniec \cite{duke-friedlander-iwaniec}. A sieve is used to deduce results about (almost) prime arguments and values. Our notation and the structure of the proof follow that of \cite{DI}.

Let $P(t) = At^2 + Bt + C$. Let $S$ denote the set of integers with no prime factor less than $x^c$, except possibly having prime factors which divide $C(B^2 - 4AC)$. We consider the slightly relaxed problem of considering integer solutions to the equation
\begin{align}
\label{eq:equ}
P(m_1) = m_2m_3
\end{align}
with $m_1 \in [1, x] \cap S$ and $m_2 \in [x^{1-\epsilon}, x^{1+\epsilon}] \cap S$. We will obtain that the number of solutions to this equation is $\ll_c \epsilon \pi(x)$, which implies the result.

We will apply a sieve to control the set $S$, and hence we first consider the problem with the condition $m_1, m_2 \in S$ replaced by divisibility conditions $d_1 \mid m_1, d_2 \mid m_2$, where $d_1, d_2$ are coprime with $C(B^2 - 4AC)$ and less than $x^{c}$. We will obtain a power-saving error term (as long as $c$ is small enough, which we may assume), which is enough for applying the sieve. Note that it suffices to consider the case $(d_1, d_2) = 1$ and $\omega(d_2) \neq 0$, where $\omega(m)$ is the number of solutions to $P(x) \equiv 0 \pmod{m}$, as otherwise the number of solutions to \eqref{eq:equ} with $d_i \mid m_i$ is zero.

Let $b_1, b_2, \ldots$ be $C^{\infty}$ functions with $\sum b_j(t) = 1$ for all $t \in \mathbb{R}$ such that $b_j$ is supported on the interval $[M_{j-1}, M_{j+1}] := [2^{j-1}, 2^{j+1}]$ and such that $b_j^{(n)}(t) \ll_n M_j^{-n}$ for any $n$, where $f^{(n)}$ is the $n$th derivative of $f$.

The number of solutions to \eqref{eq:equ} with $d_i \mid m_i, m_1 \le x$ and $m_2 \in [x^{1-\epsilon}, x^{1+\epsilon}]$ may be bounded from above by
\begin{align*}
\sum_{\substack{j_1, j_2, j_3 \\ \text{Supp } b_{j_2} \cap \ [x^{1 - \epsilon}, x^{1 + \epsilon}] \neq \emptyset}} \sum_{\substack{P(m_1) = m_2m_3 \\ d_1 \mid m_1 \\ d_2 \mid m_2 \\ m_1 \le x}} b_{j_1}(m_1)b_{j_2}(m_2)b_{j_3}(m_3).
\end{align*}
Fix $j_1, j_2, j_3$ and consider the inner sum. For brevity, we abuse notation and write $b_1, b_2, b_3$ and $M_1, M_2, M_3$ instead of $b_{j_1}, b_{j_2}, b_{j_3}$ and $M_{j_1}, M_{j_2}, M_{j_3}$.

Write the sum as (noting that any solution satisfies $(d_1, m_2) = 1$)
\begin{align*}
\sum_{\substack{m_2 \\ d_2 \mid m_2 \\ (d_1, m_2) = 1}} b_2(m_2) \sum_{\substack{v \pmod{d_1m_2} \\ P(v) \equiv 0 \pmod{m_2} \\ d_1 \mid v}} \sum_{\substack{m_1 \le x \\ m_1 \equiv v \pmod{d_1m_2}}} b_1(m_1)b_3\left(\frac{P(m_1)}{m_2}\right),
\end{align*}
and use the Poisson summation formula to obtain
\begin{align*}
\sum_{\substack{m_2 \\ d_2 \mid m_2 \\ (d_1, m_2) = 1}} \frac{b_2(m_2)}{d_1m_2} \sum_{\substack{v \pmod{d_1m_2} \\ P(v) \equiv 0 \pmod{m_2} \\ d_1 \mid v}} \sum_{h \in \mathbb{Z}} e\left(\frac{-hv}{d_1m_2}\right) \int_0^x b_1(t)b_3\left(\frac{P(t)}{m_2}\right)e\left(\frac{h}{d_1m_2}t\right) dt.
\end{align*}
The term $h = 0$ will give the main term, which will turn out to be roughly of order $x/d_1d_2$. Before further evaluating the main term we bound the error terms arising from $h \neq 0$. Partial integration for $h \neq 0$ gives
\begin{align*}
\sum_{\substack{m_2 \\ d_2 \mid m_2}} b_2(m_2)\sum_{\substack{v \pmod{d_1m_2} \\ P(v) \equiv 0 \pmod{m_2} \\ d_1 \mid v}} \sum_{h \neq 0} \frac{1}{2\pi i h}e\left(\frac{-hv}{d_1m_2}\right) \times \\
\left(b_1(x)b_3\left(\frac{P(x)}{m_2}\right)e\left(\frac{hx}{d_1m_2}\right) - \int_0^x \left(b_1(t)b_3\left(\frac{P(t)}{m_2}\right)\right)'e\left(\frac{ht}{d_1m_2}\right) dt\right).
\end{align*}
We show how to bound the part of the sum involving $b_1(x)b_2(P(x)/m_2)e(hx/d_1m_2)$; the part involving the integral is bounded similarly (by first changing the order of summation and integration and then performing similar considerations).

The contribution of large $h$ is treated as follows (cf. \cite[Section 4]{DI}). One has for any $H$
\begin{align*}
\sum_{|h| > H} \frac{1}{2\pi i h} e\left(\frac{h(x - v)}{d_1m_2}\right) = O((1 + H||(x-v)/d_1m_2||)^{-1}),
\end{align*}
where $||t||$ denotes the distance to the nearest integer. Let $\delta \in (0, 1/2]$ be given, and consider the number $f(\delta)$ of pairs $m_2, v$ with $m_2 \in \text{Supp } b_2 \subset [M_2/2, 2M_2]$ and $0 \le v < d_1m_2$ for which one has $||(x-v)/d_1m_2|| < \delta$. These pairs satisfy the relaxed condition $||(x-v)/m_2|| < d_1\delta$ as well, and so
\begin{align*}
\Big|\Big|\frac{Ax^2 + Bx + C}{m_2}\Big|\Big| = \\
\Big|\Big|\frac{Ax^2 + Bx + C - (Av^2 + Bv + C)}{m_2}\Big|\Big| = \\
\Big|\Big|\frac{(x-v)(Ax + Av + B)}{m_2}\Big|\Big| < \\
(Ax + Av + B)d_1\delta = O((x+M_2)d_1^2\delta).
\end{align*}
It follows that for any such pair $m_2, v$ there exist integers $y, z$ with $|y| = O((x/M_2+1)d_1^2\delta)$ such that
$$\lfloor P(x) \rfloor + zm_2 - y = 0.$$
Given $y$, any value of $m_2$ must be such that
$$m_2 \mid \lfloor P(x) \rfloor - y,$$
and hence there are at most $x^{\epsilon}$ such values for any $\epsilon > 0$. (Note that $y = O(x^{1.5})$, say, so $\lfloor P(x) \rfloor - y \neq 0$.) It follows that
$$f(\delta) = O((x/M_2 + 1)x^{2c + \epsilon}\delta)$$
for any $\epsilon > 0$. By dyadic summation it then follows that 
\begin{align*}
\sum_{\substack{m_2 \\ d_2 \mid m_2 \\ (d_1, m_2) = 1}} b_2(m_2) \sum_{\substack{v \pmod{d_1m_2} \\ P(v) \equiv 0 \pmod{m_2} \\ d_1 \mid v}} \sum_{|h| > H} \frac{1}{2\pi i h} e\left(\frac{-hv}{d_1m_2}\right)b_1(x)b_3\left(\frac{P(x)}{m_2}\right)e\left(\frac{hx}{d_1m_2}\right) = \\
\sum_{\substack{m_2 \\ d_2 \mid m_2 \\ (d_1, m_2) = 1}} b_2(m_2)b_1(x)b_3\left(\frac{P(x)}{m_2}\right) \sum_{\substack{v \pmod{d_1m_2} \\ P(v) \equiv 0 \pmod{m_2} \\ d_1 \mid v}} \sum_{|h| > H} \frac{1}{2\pi i h} e\left(\frac{h(x-v)}{d_1m_2}\right) \ll \\
\sum_{j = 0}^{\log_2(H)} f(2^j/H) \cdot \frac{1}{2^j} \ll \\
\frac{1}{H}(x/M_2 + 1)x^{2c + o(1)},
\end{align*}
as long as $H$ is $\ll x^{O(1)}$, say. Hence, if $c$ and $\epsilon$ are small enough and $H$ is larger than $x^{10c}$ (say), then the error arising from $|h| > H$ is small enough.

The contribution of small $h$ is treated via Weyl sums. We have the task of estimating, for each fixed $h$,
\begin{align*}
\sum_{\substack{m_2 \\ d_2 \mid m_2 \\ (d_1, m_2) = 1}} b_2(m_2)e\left(\frac{hx}{d_1m_2}\right)b_3\left(\frac{P(x)}{m_2}\right) \sum_{\substack{v \pmod{m_2} \\ P(d_1v) \equiv 0 \pmod{m_2}}} e\left(\frac{-vh}{m_2}\right).
\end{align*}
Note that we may drop the condition $(d_1, m_2) = 1$, as otherwise the inner sum is empty. Define
$$f(m_2) = f_{d_1, h, x}(m_2) = b_2(m_2)e\left(\frac{hx}{d_1m_2}\right)b_3\left(\frac{P(x)}{m_2}\right)$$
so that our sum may be written as
\begin{align*}
\sum_{\substack{m_2 \\ d_2 \mid m_2}} f(m_2) \sum_{\substack{v \pmod{m_2} \\ P(d_1v) \equiv 0 \pmod{m_2}}} e\left(\frac{-vh}{m_2}\right).
\end{align*}
By Abel summation, the above is bounded in absolute value by
\begin{align*}
\int_{\text{Supp } f} |f'(t)|\Big|\sum_{\substack{m_2 \le t \\ d_2 \mid m_2}} \sum_{\substack{v \pmod{m_2} \\ P(d_1v) \equiv 0 \pmod{m_2}}} e\left(\frac{-vh}{m_2}\right)\Big| dt
\end{align*}
For the sum over $m_2$ and $t$ we apply \cite[Proposition 1]{duke-friedlander-iwaniec} to obtain a power-saving upper bound $t^{1 - \kappa}$ for some $\kappa > 0$.

We should note that in \cite[Proposition 1]{duke-friedlander-iwaniec} the dependency in the discriminant $B^2 - 4AC$ of $P$ is not specified. However, one can check that the dependency is polynomial in $D$, which is sufficient for us. See for example \cite{duke-friedlander-iwaniec2} for more precise formulations in the discriminant aspect. We remark that in many results in this area (such as \cite[Proposition 1]{duke-friedlander-iwaniec} and \cite[Theorem 1]{duke-friedlander-iwaniec2}) one further imposes conditions on the sign of the discriminant and/or assumes that the discriminant is fundamental, which is why we have restricted to negative discriminant. However, at least in \cite{toth} the methods of \cite{duke-friedlander-iwaniec} are extended to arbitrary discriminants, so such issues can be worked around.

For the derivative we calculate
\begin{align*}
f'(t) \ll \frac{1}{t} + \frac{hx}{d_1t^2} + \frac{1}{t} \ll \frac{1}{t}\left(1 + \frac{hx}{t}\right).
\end{align*}
As $f$ is supported on $[M_2/2, 2M_2]$ and $x^{1 - \epsilon} \ll M_2 \ll x^{1 + \epsilon}$, it follows that (for small enough $c$ and $\epsilon$) the integral is
\begin{align*}
\ll M_2x^{-\kappa'}
\end{align*}
for some small constant $\kappa' > 0$.

The results for large and small $h$ together yield the existence of a constant $\kappa'' > 0$ such that
\begin{align*}
\sum_{j_1, j_3} \sum_{\substack{P(m_1) = m_2m_3 \\ d_1 \mid m_1, d_2 \mid m_2 \\ m_1 \le x}} b_{j_1}(m_1)b_{j_2}(m_2)b_{j_3}(m_3) = \frac{x}{d_1}\sum_{\substack{m_2 \\ d_2 \mid m_2 \\ (d_1, m_2) = 1}} \frac{b_{j_2}(m_2)\omega(m_2)}{m_2} + O(x^{1 - \kappa''}),
\end{align*}
where $\omega(m)$ denotes the number of solutions to $P(v) \equiv 0 \pmod{m}$. The error term is small enough to be neglected.

To evaluate the main term, we proceed by first extracting the condition $(d_1, m_2) = 1$ by Möbius inversion to get (neglecting the error term)
\begin{align*}
\frac{x}{d_1} \sum_{k \mid d_1} \mu(k) \sum_{\substack{m_2 \\ d_2k \mid m_2}} \frac{b_{j_2}(m_2)\omega(m_2)}{m_2}.
\end{align*}
The sum has been calculated in \cite[Section 8]{DI2} in the special case $P(x) = x^2 + 1$, $d_1 = 1$. The general case follows similarly. Here is a brief overview highlighting the differences. By Hensel's lemma, one has $\omega(p^k) = \omega(p)$ for any $p \nmid B^2 - 4AC$, from which one sees that for any $d$ with $\omega(d) \neq 0, (d, B^2 - 4AC) = 1$ we have
\begin{align*}
\sum_{m = 1}^{\infty} \frac{\omega(dm)}{\omega(d)}m^{-s} = \frac{\zeta(s)L(s, \chi_{\text{QR}})}{\zeta(2s)} \prod_{p \mid d} \left(1 + \frac{1}{p^s}\right)^{-1} \times \\
\prod_{p \mid B^2 - 4AC} \frac{(1 + \omega(p)p^{-s} + \omega(p^2)p^{-2s} + \ldots)(1 + p^{-2s} + p^{-4s} + \ldots)}{1 + p^{-s} + p^{-2s} + \ldots}
\end{align*}
where $\chi_{QR}$ is a character modulo $|B^2 - 4AC|$ such that $1 + \chi_{QR}(p) = \omega(p)$ for any prime $p$. Such a character exists by quadratic reciprocity. Then, one expresses $b_2(m_2)/m_2$ as an integral arising from the Mellin inversion and applies the Cauchy residue theorem to move the line of integration. The main term arises from the pole of $\zeta(s)$ at $s = 1$, while the error term is power-saving (and thus negligible). Our sum is thus, up to error,
\begin{align*}
\frac{x\omega(d_2)}{d_1d_2}\prod_{p \mid d_2} \left(1 + \frac{1}{p}\right)^{-1} \sum_{k \mid d_1} \frac{\mu(k)\omega(k)}{k} \prod_{p \mid k} \left(1 + \frac{1}{p}\right)^{-1} \frac{L(1, \chi_{\text{QR}})}{\zeta(2)} \int \frac{b_{j_2}(t)}{t} dt \times \\
\prod_{p \mid B^2 - 4AC} \frac{(1 + \omega(p)p^{-1} + \ldots)(1 + p^{-2} + \ldots)}{1 + p^{-1} + \ldots}
\end{align*}
(It is not hard to prove that the last product is finite.)

Note now that the main term is multiplicative in $d_1$ and $d_2$, with value in $(d_1, d_2) = (p, 1)$ equal to $\frac{1}{p}(1 - \omega(p)/(p+1))$ and in $(d_1, d_2) = (1, p)$ to $\omega(p)/(p+1)$. It follows that after applying a sieve (as in \cite[Section 8]{DI2}) we obtain that the number of solutions to \eqref{eq:equ} with $m_1 \le x$ and $m_1, m_2$ coprime with any prime $p \le x^c, (p, C(B^2 - 4AC)) = 1$ is, after weighting by $b_{j_2}(m_2)$, bounded by a constant times $\ll x/(\log x)^2$. Summing over the $O(\epsilon \log x)$ values of $j_2$ with $\text{Supp } b_{j_2} \cap [x^{1 - \epsilon}, x^{1 + \epsilon}] \neq \emptyset$ gives the desired upper bound $\ll \epsilon x/\log x$.
\end{proof}

\section{Almost equidistribution of linear recurrences modulo primes}

We prove Theorem \ref{thm:linRec}. With the setup of the theorem, let $\epsilon$ be such that $n = \lfloor 1/\epsilon \rfloor$. We apply Theorem \ref{thm:Nk}. It follows that for a lower density $d \ge 1 - c_k\epsilon$ of primes $p \in S$ one has
$$\ord_p(\varphi(\alpha_i)) \ge p^{g(k) + \epsilon}.$$
By \cite[Corollary 4.2]{niederreiter} we then have
\begin{align*}
\Big||\{1 \le j \le T_{i, p} : a_{i, j} \equiv r \pmod{p}\}| - \frac{T_{i, p}}{p}\Big| \le p^{k/2},
\end{align*}
and so
\begin{align*}
|\{1 \le j \le T_{i, p} : a_{i, j} \equiv r \pmod{p}\}| = \frac{T_{i, p}}{p}\left(1 + O(p^{-\epsilon})\right).
\end{align*}
which is the desired result.

We then show how to deduce Corollary \ref{cor:linRec}. Fix some small $\epsilon > 0$.

Let $\mathcal{P}$ denote the set of primes $p$ such that for any $d \mid k, d > 2$ the number $\Phi_d(p)$ is not divisible by any number $m \in [p^{1-\epsilon^2}, p^{1+\epsilon^2}]$, for which the smallest prime factor of $m$ is larger than $p^{1/2k}$. The lower density of $\mathcal{P}$ with respect to the set of all primes is $1 - O(\epsilon^2)$ by Proposition \ref{prop:mediumPrimeFactor}.

For an algebraic number $\alpha$ of degree $k$, let $S(\alpha)$ denote the set of primes $p$ for which the reductions of $\alpha$ to $\mathbb{F}_{p^{\infty}}$ have degree $k$. Let $\mathcal{A}$ denote the set of algebraics $\alpha$ of degree $k$ which are multiplicatively independent with their conjugates and for which $S(\alpha)$ is infinite. Note that by the Chebotarev density theorem $S(\alpha)$ is infinite if and only if its (lower) density is at least $1/k!$. Hence for $\epsilon$ small enough (in terms of $k$) the set $S(\alpha)$ has positive lower density if and only if the set $S'(\alpha) = S(\alpha) \cap \mathcal{P}$ has positive lower density.

Let $T'(\alpha)$ denote the set of $p \in S'(\alpha)$ for which $\ord_p(\varphi(\alpha))$ is less than $p^{g(k) + \epsilon^3}$. Say that $\alpha \in \mathcal{A}$ is \emph{bad} if the lower density of $T'(\alpha)$ in $S'(\alpha)$ is greater than $\epsilon/2$.

Note that if $\alpha \in \mathcal{A}$ is such that the conclusion of the theorem is not true for $\alpha$ when $c_{\epsilon} = \epsilon^3$, then (by the proof of Theorem \ref{thm:linRec}) $\alpha$ is bad.

Assume now that the result of the corollary is not true, so that there exist arbitrarily large sets of bad $\alpha_1, \ldots , \alpha_N \in \mathcal{A}$ such that the splitting fields of $\alpha_i$ are linearly disjoint. Note that, by Theorem \ref{thm:Nk}, there exists an integer $n$ (independent of $N$ and $\alpha_i$) such that for any subset $\{\alpha_{i_1}, \ldots , \alpha_{i_m}\}$ of any size $m$, the density of $p \in S'(\alpha_{i_1}) \cap \ldots \cap S'(\alpha_{i_m})$ for which $p \in T'(\alpha_{i_j})$ for at least $n$ values of $j$ is zero.

After noting that for any $\alpha \in \mathcal{A}$ the lower density of $S'(\alpha)$ is bounded from below by some constant $c > 0$ independent of $\alpha$, it follows that for large values of $x$ we have
\begin{align*}
Nc|\mathcal{P} \cap [1, x]|\epsilon/3 \le \\
\sum_{i = 1}^N |T'(\alpha_i) \cap [1, x]| = \\
\sum_{\substack{p \in \mathcal{P} \\ p \le x}} |\{1 \le i \le n : p \in T'(\alpha_i)| \le \\
o(\pi(x)) +  (n-1)\pi(x),
\end{align*}
which results in a contradiction for $N$ large enough.

\begin{remark}
\label{rem:lowerDensity}
Consider the following hypothetical situation: The orders of $\alpha \in \mathcal{A}$ are almost always of almost maximal order, except that for each $\alpha \in \mathcal{A}$ there is a sequence of intervals $I_1(\alpha), I_2(\alpha), \ldots$ of the form $[y, y^2]$ such that the order of $\alpha$ in $\mathbb{F}_{p^k}$ is less than $p^{g(k)}$ for $p \in I_i(\alpha)$ (but larger than $p^{g(k) - \epsilon}$ for any fixed $\epsilon > 0$). One can construct the intervals $I_i(\alpha)$ such that this situation is consistent with Theorem \ref{thm:Nk} (by taking the intervals to be far away from each other and disjoint for algebraics which, together with their conjugates, are multiplicatively independent). However, in this situation the result of Corollary \ref{cor:linRec} is not true with upper density replaced by lower density, as the lower density is zero.

A similar remark applies to Heath-Brown's \cite{heath-brown} classical result on the original primitive root conjecture. While higher rank inspections are sufficient for proving that almost any integer is a primitive root modulo $p$ for infinitely many primes $p$, it could, hypothetically, be the case that for each element there exist very long intervals such that the element fails to satisfy the given property for primes in these intervals, and thus one cannot obtain lower density results this way.
\end{remark}

\bibliography{higherDegreeArtinConjecture}
\bibliographystyle{plain}

\end{document}